\documentclass[12pt,leqno,draft]{article}
\usepackage[pagewise]{lineno}
\usepackage{amsfonts}
\usepackage{amsmath, amsthm, amsfonts, amssymb, color}
\usepackage{mathrsfs}
\usepackage{color}
\usepackage{stmaryrd}
\newtheorem{thm}{Theorem}[section]

\newtheorem{lem}[thm]{Lemma}

\newtheorem{rem}[thm]{Remark}
\theoremstyle{definition}
\newtheorem{defn}{Definition}[section]
\newcommand{\scr}[1]{\mathscr #1}
\definecolor{wco}{rgb}{0.5,0.2,0.3}

\numberwithin{equation}{section} \theoremstyle{remark}

\newcommand{\ua}{\uparrow}

\title{{\bf Exponential Ergodicity and Propagation of Chaos for Path-Distribution Dependent Stochastic Hamiltonian System}\footnote{Supported in
 part by  NNSFC (12271398).}
}
\author{
{\bf     Xing Huang $^{a)}$, Wujun Lv $^{b)}$   }\\
\footnotesize{  a) Center for Applied Mathematics, Tianjin University, Tianjin 300072, China}\\
\footnotesize{  xinghuang@tju.edu.cn}\\
\footnotesize{b)        Department of Statistics, College of Science, Donghua University, Shanghai, 201620, China}\\
\footnotesize{  lvwujun@dhu.edu.cn}}
\begin{document}
\allowdisplaybreaks
\def\R{\mathbb R}  \def\ff{\frac} \def\ss{\sqrt} \def\B{\mathbf
B}
\def\N{\mathbb N} \def\kk{\kappa} \def\m{{\bf m}}
\def\ee{\varepsilon}\def\ddd{D^*}
\def\dd{\delta} \def\DD{\Delta} \def\vv{\varepsilon} \def\rr{\rho}
\def\<{\langle} \def\>{\rangle} \def\GG{\Gamma} \def\gg{\gamma}
  \def\nn{\nabla} \def\pp{\partial} \def\E{\mathbb E}
\def\d{\text{\rm{d}}} \def\bb{\beta} \def\aa{\alpha} \def\D{\scr D}
  \def\si{\sigma} \def\ess{\text{\rm{ess}}}
\def\beg{\begin} \def\beq{\begin{equation}}  \def\F{\scr F}
\def\Ric{\text{\rm{Ric}}} \def\Hess{\text{\rm{Hess}}}
\def\e{\text{\rm{e}}} \def\ua{\underline a} \def\OO{\Omega}  \def\oo{\omega}
 \def\tt{\tilde} \def\Ric{\text{\rm{Ric}}}
\def\cut{\text{\rm{cut}}} \def\P{\mathbb P} \def\ifn{I_n(f^{\bigotimes n})}
\def\C{\scr C}   \def\G{\scr G}   \def\aaa{\mathbf{r}}     \def\r{r}
\def\gap{\text{\rm{gap}}} \def\prr{\pi_{{\bf m},\varrho}}  \def\r{\mathbf r}
\def\Z{\mathbb Z} \def\vrr{\varrho} \def\ll{\lambda}
\def\L{\scr L}\def\Tt{\tt} \def\TT{\tt}\def\II{\mathbb I}
\def\i{{\rm in}}\def\Sect{{\rm Sect}}  \def\H{\mathbb H}
\def\M{\scr M}\def\Q{\mathbb Q} \def\texto{\text{o}} \def\LL{\Lambda}
\def\Rank{{\rm Rank}} \def\B{\scr B} \def\i{{\rm i}} \def\HR{\hat{\R}^d}
\def\to{\rightarrow}\def\l{\ell}\def\iint{\int}
\def\EE{\scr E}\def\no{\nonumber}
\def\A{\scr A}\def\V{\mathbb V}\def\osc{{\rm osc}}
\def\BB{\scr B}\def\Ent{{\rm Ent}}\def\3{\triangle}\def\H{\scr H}
\def\U{\scr U}\def\8{\infty}\def\1{\lesssim}\def\HH{\mathrm{H}}
 \def\T{\scr T}
 \def\R{\mathbb R}  \def\ff{\frac} \def\ss{\sqrt} \def\B{\mathbf
B} \def\W{\mathbb W}
\def\N{\mathbb N} \def\kk{\kappa} \def\m{{\bf m}}
\def\ee{\varepsilon}\def\ddd{D^*}
\def\dd{\delta} \def\DD{\Delta} \def\vv{\varepsilon} \def\rr{\rho}
\def\<{\langle} \def\>{\rangle} \def\GG{\Gamma} \def\gg{\gamma}
  \def\nn{\nabla} \def\pp{\partial} \def\E{\mathbb E}
\def\d{\text{\rm{d}}} \def\bb{\beta} \def\aa{\alpha} \def\D{\scr D}
  \def\si{\sigma} \def\ess{\text{\rm{ess}}}
\def\beg{\begin} \def\beq{\begin{equation}}  \def\F{\scr F}
\def\Ric{\text{\rm{Ric}}} \def\Hess{\text{\rm{Hess}}}
\def\e{\text{\rm{e}}} \def\ua{\underline a} \def\OO{\Omega}  \def\oo{\omega}
 \def\tt{\tilde} \def\Ric{\text{\rm{Ric}}}
\def\cut{\text{\rm{cut}}} \def\P{\mathbb P} \def\ifn{I_n(f^{\bigotimes n})}
\def\C{\scr C}      \def\aaa{\mathbf{r}}     \def\r{r}
\def\gap{\text{\rm{gap}}} \def\prr{\pi_{{\bf m},\varrho}}  \def\r{\mathbf r}
\def\Z{\mathbb Z} \def\vrr{\varrho} \def\ll{\lambda}
\def\L{\scr L}\def\Tt{\tt} \def\TT{\tt}\def\II{\mathbb I}
\def\i{{\rm in}}\def\Sect{{\rm Sect}}  \def\H{\mathbb H}
\def\M{\scr M}\def\Q{\mathbb Q} \def\texto{\text{o}} \def\LL{\Lambda}
\def\Rank{{\rm Rank}} \def\B{\scr B} \def\i{{\rm i}} \def\HR{\hat{\R}^d}
\def\to{\rightarrow}\def\l{\ell}\def\iint{\int}
\def\EE{\scr E}\def\Cut{{\rm Cut}}
\def\A{\scr A} \def\Lip{{\rm Lip}}
\def\BB{\scr B}\def\Ent{{\rm Ent}}\def\L{\scr L}
\def\R{\mathbb R}  \def\ff{\frac} \def\ss{\sqrt} \def\B{\mathbf
B}
\def\N{\mathbb N} \def\kk{\kappa} \def\m{{\bf m}}
\def\dd{\delta} \def\DD{\Delta} \def\vv{\varepsilon} \def\rr{\rho}
\def\<{\langle} \def\>{\rangle} \def\GG{\Gamma} \def\gg{\gamma}
  \def\nn{\nabla} \def\pp{\partial} \def\E{\mathbb E}
\def\d{\text{\rm{d}}} \def\bb{\beta} \def\aa{\alpha} \def\D{\scr D}
  \def\si{\sigma} \def\ess{\text{\rm{ess}}}
\def\beg{\begin} \def\beq{\begin{equation}}  \def\F{\scr F}
\def\Ric{\text{\rm{Ric}}} \def\Hess{\text{\rm{Hess}}}
\def\e{\text{\rm{e}}} \def\ua{\underline a} \def\OO{\Omega}  \def\oo{\omega}
 \def\tt{\tilde} \def\Ric{\text{\rm{Ric}}}
\def\cut{\text{\rm{cut}}} \def\P{\mathbb P} \def\ifn{I_n(f^{\bigotimes n})}
\def\C{\scr C}      \def\aaa{\mathbf{r}}     \def\r{r}
\def\gap{\text{\rm{gap}}} \def\prr{\pi_{{\bf m},\varrho}}  \def\r{\mathbf r}
\def\Z{\mathbb Z} \def\vrr{\varrho} \def\ll{\lambda}
\def\L{\scr L}\def\Tt{\tt} \def\TT{\tt}\def\II{\mathbb I}
\def\i{{\rm in}}\def\Sect{{\rm Sect}}  \def\H{\mathbb H}
\def\M{\scr M}\def\Q{\mathbb Q} \def\texto{\text{o}} \def\LL{\Lambda}
\def\Rank{{\rm Rank}} \def\B{\scr B} \def\i{{\rm i}} \def\HR{\hat{\R}^d}
\def\to{\rightarrow}\def\l{\ell}
\def\8{\infty}\def\I{1}\def\U{\scr U}
\maketitle

\begin{abstract} 

By Girsanov's thoerem and using the existing log-Harnack inequality for distribution independent SDEs, the log-Harnack inequality is derived for path-distribution dependent stochastic Hamiltonian systems. As an application, the exponential ergodicity in relative entropy is obtained by combining with transportation cost inequality. In addition, the quantitative propagation of chaos in the sense of Wasserstein distance, which together with the coupling by change of measure implies the quantitative propagation of chaos in total variation norm as well as relative entropy are obtained.
\end{abstract} \noindent
 \vskip 2cm

\section{Introduction}
The stochastic Hamiltonian system (SHS), which includes the kinetic
Fokker-Planck equation (see \cite{V}), has been extensively investigated in \cite{BWY,FF,GW,HLF,W2,WZ1,Z1,Z2} and references therein. More precisely, \cite{FF} has studied the regularity of stochastic kinetic equations; \cite{GW} investigated Bismut formula, gradient estimate and Harnack inequality for SHS by using coupling by change of measure; the derivative formula is extended to the case that the degenerate part is not linear by using Malliavin calculus in \cite{WZ1} and \cite{Z1}; moreover, \cite{Z1} derived the stochastic flows for SHS with linear degenerate part, and the diffusion only depends on the degenerate part; see also \cite{Z2} for the results on the stochastic flows with singular coefficients; we refer to \cite{W2} for the hypercontractivity for SHS. For the path-dependent SHS, the derivative formula and Harnack inequality are established in \cite{BWY}, see also \cite{HLF} for Harnack inequalities with singular drifts.

Recently, along with the application in nonlinear Fokker-Planck-Kolmogorov equations, McKean-Vlasov stochastic differential equations (SDEs), presented in \cite{M}, have gained much attention. There are plentiful results on these type SDEs, see for instance, \cite{BR1,BR2,CN,HX,L,RZ,Z5} and references therein.
In \cite{RW}, the exponential ergodicity of McKean-Vlasov SDEs in relative entropy is derived by log-Harnack inequality and transportation cost inequality. The log-Harnack inequality for non-degenerate McKean-Vlasov SDEs is investigated in \cite{FYW1} by coupling by change of measure. One can also refer to \cite{HRW} for the log-Harnack inequality of non-degenerate McKean-Vlasov SDEs with memory. In addition, there are lots of references on the well-posedness of McKean-Vlasov SDEs with singular coefficients, for instance, \cite{CN,HX,HS,L,MV,RZ,Z5} and references therein. Since in this paper we do not plan to pay much attention in the well-posedness for McKean-Vlasov SDEs with singular coefficients, we will not characterize the details of the well-posedness results in the above references and we will give the well-posedness result using the appendix in Section 5. 

To obtain the log-Harnack inequality for the path-distribution dependent SHS, we will adopt  the Girsanov's transform and combine with the existing log-Harnack  inequality in \cite{Wbook} and \cite{HLF}.

McKean-Vlasov SDEs can be viewed as the limit of the interacting particle system. The so called propagation of chaos (\cite{SZ}) means that the joint distribution of finite many particles converges to the product of the distribution of McKean-Vlasov SDEs as the number of interacting particle system tends to infinity, see \cite[Definition 4.1]{GKMPPT} for more details. For the propagation of chaos, \cite{L} obtain the convergence of the interacting particle system with non-degenerate noise in the total variation distance. In this paper,  we obtain the convergence of the interacting particle system in the sense of Wasserstein distance, total variation norm and relative entropy, see Theorem \ref{POC} below. Since $\C^{m+d}$ is an infinite dimensional space, to obtain the quantitative propagation of chaos, we assume that the coefficients are Lipschitz continuous in $\W_\theta^\Gamma$ instead of $L^\theta$-Wasserstein distance. For more results on the propagation of chaos, see \cite{AZ,BO,G,GKMPPT,JR,NT,SZ} and references therein.

The main contributions of this paper mainly include:
(1) The diffusion is degenerate.
(2) The model is assumed to be both path and distribution dependent.
(3) The quantitative propagation of chaos in the sense of total variation norm and relative entropy is obtained.

The paper is organized as follows: In Section 2, we prove the log-Harnack inequality for path-distribution dependent SHS; The exponential ergodicity in relative entropy is derived in Section 3, where the transportation cost inequality for the invariant probability measure is also investigated under the dissipative condition; In section 4, the quantitative propagation of chaos for path-distribution dependent SHS is studied. Finally, the well-posedness for general  path-distribution dependent SDEs and mean field interacting particle system is provided in Section 5.


Throughout the paper, fix a constant $r> 0$. For any $n\in\mathbb{N}^{+}$, let $\C^{n}= C([-r,0];\mathbb{R}^{n})$ be equipped with the uniform norm $\|\xi\|_\infty =:\sup_{s\in[-r,0]} |\xi(s)|$. For any $f\in C([-r,\infty);\mathbb{R}^{n})$, $t\geq 0$, define $f_t \in \C^{n}$ as $f_t(s)=f(t+s), s\in [-r,0]$, which is called the segment process. Let $\scr P(\C^{n})$ be the set of all probability measures in $\C^{n}$ equipped with the weak topology.
For $\theta\geq 1$, define
$$\scr P_\theta(\C^n) = \big\{\mu\in \scr P(\C^n): \mu(\|\cdot\|_\infty^\theta)<\infty\big\}.$$    It is well known that
$\scr P_\theta(\C^n)$ is a Polish space under the Wasserstein distance
$$\W_\theta(\mu,\nu):= \inf_{\pi\in \mathbf{C}(\mu,\nu)} \bigg(\int_{\C^n\times\C^n} \|\xi-\eta\|_\infty^\theta \pi(\d \xi,\d \eta)\bigg)^{\ff 1 {\theta}},\ \ \mu,\nu\in \scr P_{\theta}(\C^n),$$ where $\mathbf{C}(\mu,\nu)$ is the set of all couplings of $\mu$ and $\nu$.

Recall that for two probability measures $\mu,\nu$ on   some measurable space $(E,\scr E)$, the entropy and total variation norm are defined as follows:
$$\Ent(\nu|\mu):= \beg{cases} \int_E (\log \ff{\d\nu}{\d\mu})\d\nu, \ &\text{if}\ \nu\ \text{ is\ absolutely\ continuous\ with\ respect\ to}\ \mu,\\
 \infty,\ &\text{otherwise;}\end{cases}$$ and
$$\|\mu-\nu\|_{var} := \sup_{|f|\leq 1}|\mu(f)-\nu(f))|.$$ By Pinsker's inequality (see \cite{Pin}),
\beq\label{ETX} \|\mu-\nu\|_{var}^2\le 2 \Ent(\nu|\mu),\ \ \mu,\nu\in \scr P(E),\end{equation}
here $\scr P(E)$ denotes all probability measures on $(E,\scr E)$.
Throughout the paper, we will use $C$ or $c$ as a constant, the values of which may change from one place to another. For $n,k\in\mathbb{N}^+$, let $0_{n}$ and $0_{n\times k}$ denote the $n$ dimensional vector and $n\times k$ matrix with all components being $0$.
\section{Log-Harnack Inequality}
The log-Harnack inequality provides an estimate of the relative entropy for two probability measures, see for instance \cite[Theorem 1.4.2 (2)]{Wbook}. For the path-dependent SHS, the log-Harnack inequality has been established in \cite[Theorem 4.4.5]{Wbook}, see also \cite{HLF} for the case with singular drifts. \cite{HRW} studied log-Harnack inequality for path-distribution dependent SDEs with non-degenerate noise and the result is extended to the path-dependent SDEs with singular drift in \cite{HX}. Moreover, by Girsanov's transform and Young's inequality, the log-Harnack inequality is obtained in \cite{HS}, where the semi-linear SPDE with Dini continuous drift and non-degenerate noise is considered. In this section, we extend the method in \cite{HS} to the path-distribution dependent case including the path-distribution SHS. To this end, we first give a general result as follows.
\subsection{A General Result}
Let $T>r$ and $n,k\in \mathbb{N}^+$. Consider SDE on $\R^n$:
\begin{align}\label{GPSEQ00}\d X(t)=H_0(t,X_t)\d t+\Sigma(t,X_t)H(t,X_t,\L_{X_t})\d t+\Sigma(t,X_t)\d W(t),
\end{align}
where $H_0:[0,\infty)\times\scr C^n\to \mathbb{R}^n, H:[0,\infty)\times\scr C^n\times \scr P(\C^n)\to \mathbb{R}^k$, $\Sigma:[0,\infty)\times\C^n\to\mathbb{R}^{n}\otimes\mathbb{R}^{k}$ are measurable and $W(t)$ is a $k$-dimensional Brownian motion on some complete filtration probability space $(\Omega, \scr F, (\scr F_t)_{t\geq 0},\P)$.

Let $\hat{\scr P}(\C^n)$ be a subset of $\scr P(\C^n)$ containing all Dirac measures and it is equipped with some topology. Assume that \eqref{GPSEQ00} is well-posed in $\hat{\scr P}(\C^n)$, see Definition \ref{defws} and Theorem \ref{GWP} for general result on the well-posedness of path-distribution depedndent SDEs. For any $\mu_0\in\hat{\scr P}(\C^n)$, let $X^{\mu_0}_t$ be the unique solution to \eqref{GPSEQ00} with initial distribution $\mu_0$ and define
\begin{align}\label{seg}P_t f(\mu_0)=(P_t^\ast\mu_0) (f)=\E f(X^{\mu_0}_t),\ \ f\in \scr B_b(\C^n),t\geq 0.\end{align}
For any $\mu\in C([0,T];\hat{\scr P}(\C^n))$ and any $\F_0$-measurable random variable $X_0$ with $\L_{X_0}\in\hat{\scr P}(\C^n)$, suppose that the decoupled SDE
\begin{align}\label{GPS000}\d X^{X_0,\mu}(t)=H_0(t,X^{X_0,\mu}_t)\d t+\Sigma(t,X^{X_0,\mu}_t)H(t,X^{X_0,\mu}_t,\mu_t)\d t+\Sigma(t,X^{X_0,\mu}_t)\d W(t)
\end{align}
with $X^{X_0,\mu}_0=X_0$ has a unique strong solution. Note that \eqref{GPS000} reduces to a path dependent classical SDE, see \cite{HLF,WZ,Z2} and references therein for the well-posedness with singular coefficients. Let $P_t^\mu$ be the associated semigroup to \eqref{GPS000}, i.e.
$$P_t^\mu f(\xi)=\E f(X^{\xi,\mu}_t),\ \ \xi\in\C^n, f\in \scr B_b(\C^n), t\geq 0. $$
For $\nu\in C([0,T];\hat{\scr P}(\C^n))$, let
\begin{align*}
&\zeta_t^{\mu,\nu}=H(t,X^{X_0,\mu}_t,\mu_t)-H(t,X^{X_0,\mu}_t,\nu_t),\\ &R_t^{\mu,\nu}=\exp\left\{-\int_{0}^t\<\zeta^{\mu,\nu}_s,\d W(s)\>-\frac{1}{2}\int_{0}^t|\zeta^{\mu,\nu}_s|^2\d s\right\},\ \ t\in[0,T].
\end{align*}
\begin{thm}\label{log11}
Assume that for any $\mu,\nu\in C([0,T],\hat{\scr P}(\C^n))$, $\{R_t^{\mu,\nu}\}_{t\in[0,T]}$ is a martingale and $P_t^\nu$ satisfies the log-Harnack inequality, i.e. there exists a function $C:(r,\infty)\to(0,\infty)$ such that for any $f\in \scr B_{b}(\C^n)$ with $f>0$
\begin{align}\label{log00}P_t^\nu\log f(\xi)\leq \log P_t^\nu f(\eta)+ C(t)\|\xi-\eta\|_\infty^2,\ \ r<t\leq T, \xi,\eta\in\C^n.
\end{align}
Then we have
\begin{align}\label{lognn}P_t\log f(\nu_0)\leq \log P_tf(\mu_0) + 2C(t)\W_2(\mu_0,\nu_0)^2+\log \E (R_t^{\mu,\nu})^2,\ \ r<t\leq T.
\end{align}
Consequently,
$$\frac{1}{2}\|P_t^\ast\mu_0-P_t^\ast\nu_0\|_{var}^2\leq\Ent(P_t^\ast\mu_0|P_t^\ast\nu_0)\leq 2C(t)\W_2(\mu_0,\nu_0)^2+\log \E (R_t^{\mu,\nu})^2,\ \ r<t\leq T.$$
\end{thm}
 \begin{proof}
 By \cite[Theorem 1.4.2 (2)]{Wbook} and \eqref{ETX}, it is sufficient to prove the log-Harnack inequality \eqref{lognn}.

 Let $\mu_t=P_t^\ast \mu_0$ and $\nu_t=P_t^\ast \nu_0$, $\bar{W}(t)=W(t)+\int_0^t\zeta_s^{\mu,\nu}\d s$, $t\in[0,T]$. Since $\{R_t^{\mu,\nu}\}_{t\in[0,T]}$ is a martingale, it follows from Girsanov's theorem that $\{\bar{W}(t)\}_{t\in[0,T]}$ is a $k$-dimensional Brownian motion under $\Q_T=R_T^{\mu,\nu}\P$.
 So, \eqref{GPS000} can be rewritten as
\begin{align*}
\d X^{X_0,\mu}(t)=H_0(t,X^{X_0,\mu}_t)+\Sigma(t,X^{X_0,\mu}_t)H(t,X^{X_0,\mu}_t,\nu_t)\d t+\Sigma(t,X^{X_0,\mu}_t)\d\bar{ W}(t),\ \ X^{X_0,\mu}_0=X_0.
\end{align*}
 Letting $\bar{\mu}_t=\L_{X^{X_0,\mu}_t}|\Q_T$ and noting that $\{R_t^{\mu,\nu}\}_{t\in[0,T]}$ is a martingale, we derive
\begin{align*}\bar{\mu}_t(f)=\E^{\Q_T}f(X^{X_0,\mu}_t)=\E(R_t^{\mu,\nu}f(X^{X_0,\mu}_t)),\ \ f\in\B_b(\C^n), t\in[0,T],
\end{align*}
which implies that $\P$-a.s.
\begin{align*}\frac{\d\bar{\mu}_t}{\d\mu_t}(X^{X_0,\mu}_t)=\E(R_t^{\mu,\nu}|X^{X_0,\mu}_t),\ \ t\in[0,T].
\end{align*}
By Jensen's inequality for conditional expectation, we get
\begin{align}\label{RTQ}
\bar{\mu}_t\left(\frac{\d \bar{\mu}_t}{\d \mu_t}\right)&=\E(\E(R_t^{\mu,\nu}|X^{X_0,\mu}_t)^2)\leq \E(R_t^{\mu,\nu})^2,\ \ t\in[0,T].
\end{align}
On the other hand, taking expectation in \eqref{log00} with respect to any $\pi\in\C(\nu_0,\mu_0)$ and using Jensen's inequality and then taking infimum in $\pi\in\C(\nu_0,\mu_0)$, we get
$$(P_t^\ast\nu_0)(\log f)\leq \log \bar{\mu}_t(f)+ C(t)\W_2(\mu_0,\nu_0)^2,\ \ r<t\leq T.$$
This together with \cite[Theorem 1.4.2 (2)]{Wbook} implies that
$$\mathrm{Ent}(P_t^\ast\nu_0|\bar{\mu}_t)=\bar{\mu}_t\left(\frac{\d P_t^\ast\nu_0}{\d \bar{\mu}_t}\log\frac{\d P_t^\ast\nu_0}{\d \bar{\mu}_t}\right)\leq C(t)\W_2(\mu_0,\nu_0)^2.$$
It follows from Young's inequality and \eqref{RTQ} that
\begin{align*}P_t\log f(\nu_0)
&=\mu_t\left(\frac{\d \bar{\mu}_t}{\d \mu_t}\frac{\d P_t^\ast\nu_0}{\d \bar{\mu}_t}\log f\right)\\
&\leq \log P_tf(\mu_0)+\mu_t\left(\frac{\d \bar{\mu}_t}{\d \mu_t}\frac{\d P_t^\ast\nu_0}{\d \bar{\mu}_t}\log\left(\frac{\d \bar{\mu}_t}{\d \mu_t}\frac{\d P_t^\ast\nu_0}{\d \bar{\mu}_t}\right)\right)\\
&= \log P_t f(\mu_0)+\bar{\mu}_t\left(\frac{\d P_t^\ast\nu_0}{\d \bar{\mu}_t}\log\frac{\d \bar{\mu}_t}{\d \mu_t}\right)+\bar{\mu}_t\left(\frac{\d P_t^\ast\nu_0}{\d \bar{\mu}_t}\log\frac{\d P_t^\ast\nu_0}{\d \bar{\mu}_t}\right)\\
&\leq \log P_t f(\mu_0)+\log \bar{\mu}_t\left(\frac{\d \bar{\mu}_t}{\d \mu_t}\right)+2\bar{\mu}_t\left(\frac{\d P_t^\ast\nu_0}{\d \bar{\mu}_t}\log\frac{\d P_t^\ast\nu_0}{\d \bar{\mu}_t}\right)\\
&\leq \log P_tf(\mu_0)+\log\E(R_t^{\mu,\nu})^2+2C(t)\W_2(\mu_0,\nu_0)^2.
\end{align*}
Therefore, we complete the proof.
 \end{proof}
\subsection{Log-Harnack Inequality and Regularity for Path-Distribution Dependent SHS}
Let $m,d\in\mathbb{N}^{+}$.
In this section, consider the following path-distribution dependent stochastic Hamiltonian system on $\mathbb{R}^{m+d}$:
\beq\label{EH}
\begin{cases}
\d X(t)=\{AX(t)+MY(t)\}\d t, \\
\d Y(t)=\{Z(X(t),Y(t),\L_{(X_t,Y_t)})+B(X_t,Y_t,\L_{(X_t,Y_t)})\}\d t+\sigma\d W(t),
\end{cases}
\end{equation}
where $W=(W(t))_{t\geq 0}$ is a $d$-dimensional standard Brownian motion with respect to a complete filtration probability space $(\OO, \F, \{\F_{t}\}_{t\ge 0}, \P)$, $A$ is an $m\times m$ matrix, $M$ is an $m\times d$ matrix, $\sigma$ is a $d\times d$ matrix, $Z:\mathbb{R}^{m+d}\times \scr P(\C^{m+d})\to \mathbb{R}^d$, $B:\C^{m+d}\times \scr P(\C^{m+d})\to\mathbb{R}^d$. We should remark that the reason why we assume that the coefficients are time independent is only to coincide with the assertion in Section 3 and the result in Theorem \ref{T3.2} can also be available in the time dependent case.
To obtain the log-Harnack inequality, we make the following assumptions:
\begin{enumerate}
\item[\bf{(A1)}] $\sigma$ is invertible.
\item[\bf{(A2)}] There exists $\theta\geq 1$ and $K_Z>0$ such that
\begin{equation*}
|Z(z,\gamma)-Z(\bar{z},\bar{\gamma})|\leq K_Z(|z-\bar{z}|+\W_\theta(\gamma,\bar{\gamma})),\ \ z,\bar{z}\in\mathbb{R}^{m+d},\gamma,\bar{\gamma}\in\scr P_\theta(\C^{m+d}).
\end{equation*}
\item[\bf{(A3)}] Let $\theta$ be in {\bf{(A2)}}. There exists a constant $K_B>0$ such that
\begin{equation*}
|B(\xi,\gamma)-B(\eta,\bar{\gamma})|\leq K_B(\|\xi-\eta\|_{\infty}+\W_\theta(\gamma,\bar{\gamma})),\ \ \xi,\eta\in\C^{m+d},\gamma,\bar{\gamma}\in\scr P_\theta(\C^{m+d}).
\end{equation*}
\item[\bf{(A4)}] There exists an integer $0\leq l\leq m-1$ such that
$$\mathrm{Rank}[M,AM,\cdots,A^lM]=m.$$
\end{enumerate}

According to Remark \ref{Rem} below, under {\bf(A1)}-{\bf(A3)}, \eqref{EH} is well-posed in $\scr P_\theta(\C^{m+d})$. Denote the solution to \eqref{EH} with $\L_{(X_0,Y_0)}=\mu_0\in\scr P_{\theta}(\C^{m+d})$ by $(X_t^{\mu_0},Y_t^{\mu_0})$.
Let $P_t$ and $P_t^\ast$ be defined in the same way as in \eqref{seg} for $(X_t^{\mu_0},Y_t^{\mu_0})$ replacing $X_t^{\mu_0}$ there.
The next result characterizes the log-Harnack inequality for \eqref{EH}.
\begin{thm}\label{T3.2} Assume
  {\bf (A1)}-{\bf (A4)} and let $t>r$. Then for any $\mu_0,\nu_0\in \scr P_\theta(\C^{m+d})$ and positive
$f\in \B_b(\C^{m+d})$,
\beg{equation*}\beg{split}
P_t\log f(\mu_0)&\leq\log P_t f(\nu_0)+C^2\int_0^t\e^{2Cs}\d s\W_\theta(\mu_0,\nu_0)^2+\Sigma(t,r,\|M\|,l)\W_2(\mu_0,\nu_0)^2,\end{split}\end{equation*}
where
\beg{equation*}\beg{split}
\nonumber\Sigma(t,r,\|M\|,l)&=C\left\{\Big(\ff 1 {(t-r)\wedge1}+ \ff{\|M\|}{(t-r)^{(4l+3)}\wedge1}\Big)+\Big(1+ \ff{\|M\|}{(t-r)^{2l+1}\wedge1}\Big)^2\right\},
\end{split}\end{equation*}
and $C>0$ is a constant. Consequently, it holds
\begin{align}\label{entro}
\nonumber
&\frac{1}{2}\|P_t^\ast\mu_0-P_t^\ast\nu_0\|_{var}^2\leq\mathrm{Ent}(P_t^\ast \mu_0|P_t^\ast \nu_0)\\
&\leq C^2\int_0^t\e^{2Cs}\d s\W_\theta(\mu_0,\nu_0)^2+\Sigma(t,r,\|M\|,l)\W_2(\mu_0,\nu_0)^2.
\end{align}
 \end{thm}
 \begin{proof}
Let $n=m+d$, $k=d$,
$$H_0(x,y)=\left(
        \begin{array}{c}
          Ax+My \\
           0_d\\
        \end{array}
      \right),\ \
 H=\sigma^{-1}(Z+B), \ \ \Sigma=\left(
                                                \begin{array}{c}
                                                  0_{m\times d} \\
                                                  \sigma \\
                                                \end{array}
                                              \right),
\ \ x\in\R^m,y\in\R^d.$$
Let $\mu_t=P_t^\ast\mu_0$ and $\nu_t=P_t^\ast\nu_0$. For simplicity, we denote $(X_s,Y_s)=(X_s^{\mu_0}, Y_s^{\mu_0})$. Set $$\zeta^{\mu,\nu}_s=\sigma^{-1}[Z(X(s),Y(s),\mu_s)+B(X_s,Y_s,\mu_s)-Z(X(s),Y(s),\nu_s)-B(X_s,Y_s,\nu_s)],$$
By {\bf(A2)}-{\bf(A3)} and Remark \ref{Rem} below, there exists a constant $C>0$ such that
\begin{align*}|\zeta^{\mu,\nu}_s|\leq \|\sigma^{-1}\|(K_Z+K_B)\W_\theta(\mu_s,\nu_s)\leq C\e^{Cs}\W_\theta(\mu_0,\nu_0),\ \ s\in[0,t].
\end{align*}
Recalling the definition of $R_t^{\mu,\nu}$ in Theorem \ref{log11}, we arrive at
$$\log \E (R_t^{\mu,\nu})^2\leq \log \mathrm{ess sup}_{\Omega}\e^{\int_0^t|\zeta^{\mu,\nu}_s|^2\d s}\leq \int_0^tC^2\e^{2Cs}\W_\theta(\mu_0,\nu_0)^2\d s.$$
On the other hand, by \cite[Theorem 4.4.5]{Wbook}, we know
\begin{align*}
P_t^\nu\log f(\xi)\leq \log P_t^\nu f(\eta)+ \Sigma(t,r,\|M\|,k)\|\xi-\eta\|_\infty^2.
\end{align*}
So, applying Theorem \ref{log11}, we complete the proof.
 \end{proof}
\section{Exponential Ergodicity}
In this section, we investigate the exponential ergodicity of \eqref{EH} in $L^2$-Wasserstein distance as well as in relative entropy.
To this end, we assume
\begin{enumerate}
\item[{\bf(C)}] There exist $\lambda_1>0, \lambda_2,\lambda_3\geq 0$ with $\lambda_2+\lambda_3< \sup_{\delta\in[0,\lambda_1]}\delta\e^{-\delta r}$ such that for any $\xi=(\xi^{(1)},\xi^{(2)}), \bar{\xi}=(\bar{\xi}^{(1)},\bar{\xi}^{(2)})\in \C^{m+d}, \gamma,\bar{\gamma}\in \scr P_{2}(\C^{m+d})$,
\begin{align*}&2\<A(\xi^{(1)}(0)-\bar{\xi}^{(1)}(0))+M(\xi^{(2)}(0)-\bar{\xi}^{(2)}(0)),\ \ \xi^{(1)}(0)-\bar{\xi}^{(1)}(0)\>, \\
&+2\<Z(\xi(0),\gamma)-Z(\bar{\xi}(0),\bar{\gamma})+B(\xi,\gamma)-B(\bar{\xi}, \bar{\gamma}),\ \ \xi^{(2)}(0)-\bar{\xi}^{(2)}(0)\>\\
&\leq -\lambda_1|\xi(0)-\bar{\xi}(0)|^2+\lambda_2\|\xi-\bar{\xi}\|_\infty^2+\lambda_3\W_2(\gamma,\bar{\gamma})^2.
\end{align*}
\end{enumerate}
\begin{thm}\label{EXPEN}
Assume ${\bf(C)}$ and {\bf (A1)}-{\bf (A4)}  with $\theta=2$.
Then $P_t^*$  has a unique  invariant probability measure $\mu^\ast\in \scr P_2(\C^{m+d})$ with
\begin{align*}&\max(\W_2(P_t^\ast\nu, \mu^\ast)^2,\mathrm{Ent}(P_t^\ast\nu|\mu^\ast))\\
&\leq c\e^{-2\kappa t}\min(\W_2(\nu, \mu^\ast)^2,\mathrm{Ent}(\nu|\mu^\ast)),\ \ \nu\in\scr P_2(\C^{m+d}),t>2r\end{align*}
for some constants $c,\kappa>0$.
\end{thm}
\begin{proof} By \cite[Remark 2.1]{HRW}, {\bf (C)} implies that there exist constants $c_0,\kappa>0$ such that
$$\W_2(P_t^\ast\mu_0,P_t^\ast \nu_0)\leq c_0\e^{-\kappa t}\W_2(\mu_0,\nu_0),\ \ \mu_0,\nu_0\in\scr P_2(\C^{m+d}),t>0.$$
Then it is standard to prove that $P_t^*$  has a unique  invariant probability measure $\mu^\ast\in \scr P_2(\C^{m+d})$ with
\begin{align}\label{WCO}\W_2(P_t^\ast\nu, \mu^\ast)^2\leq c_0^2\e^{-2\kappa t}\W_2(\nu, \mu^\ast)^2, \ \ \nu\in\scr P_2(\C^{m+d}),t>0.\end{align}
Combining this with \eqref{entro} for $t=2r$ and \eqref{TCI} below, we complete the proof by using \cite[Theorem 2.1]{RW}.
\end{proof}
\subsection{Transportation cost inequality}
To obtain the exponential ergodicity in relative entropy, we also need to prove the transportation cost inequality for $\mu^\ast$. \cite{BWY13} give a proof of transportation cost inequality for the solution to path dependent SDEs starting from dirac measure and the technique used there is also available in the present case. Furthermore, under the dissipative condition {\bf(C)}, we can derive a uniform constant with respect to time variable $T$ in the transportation cost inequality for the solution to \eqref{EH} on $[0,T]$ starting from dirac measure, see \eqref{TCP} below. Then applying \cite[Lemma 2.1]{DGW} and \cite[Lemma 2.2]{DGW}, the stability of transportation cost inequality, $\mu^\ast$ satisfies the transportation cost inequality due to \eqref{WCO}.
\begin{thm}\label{TC3} Assume {\bf(C)}. Then the transportation cost  inequality holds for the invariant probability measure $\mu^\ast$, i.e.
\begin{align}\label{TCI}\W_2(\nu, \mu^\ast)^2\leq 2\e^{(\lambda_1-\epsilon)r} \frac{\|\sigma\|}{\epsilon}\mathrm{Ent}(\nu|\mu^\ast), \ \ \nu\in\scr P_2(\C^{m+d})\end{align}
with some constant $\epsilon\in(0,\lambda_1)$.
\end{thm}
\begin{proof} Consider
\beq\label{EH000}
\begin{cases}
\d X(t)=\{AX(t)+MY(t)\}\d t, \\
\d Y(t)=\{Z(X(t),Y(t),\mu^\ast)+B(X_t,Y_t,\mu^\ast)\}\d t+\sigma\d W(t).
\end{cases}
\end{equation}
For any $\xi\in\scr C^{m+d}$, $(X_t^{\xi,\mu^\ast},Y_t^{\xi,\mu^\ast})$ be the unique solution to \eqref{EH000} with initial value $\xi$. Let $P_t^{\mu^\ast}(\xi,\d\eta)=\L_{(X_t^{\xi,\mu^\ast},Y_t^{\xi,\mu^\ast})}(\d \eta)$.
According to {\bf(C)} and \cite[Remark 2.1]{HRW}, $\mu^\ast$ is the unique invariant probability measure of \eqref{EH000} and there exist constants $\tilde{c},\tilde{\kappa}>0$ such that
\begin{align}\label{ECR}\W_2(P_t^{\mu^\ast}(\xi,\cdot), \mu^\ast)\leq \tilde{c}\e^{-\tilde{\kappa} t}\W_2(\delta_\xi, \mu^\ast).
\end{align}
As in the proof of \cite[Lemma 2.2]{BWY13}, denote by $\Pi_\xi^T$ as the distribution of $(X_t^{\xi,\mu^\ast},Y_t^{\xi,\mu^\ast})_{t\in[0,T]}$.
Define the distance $$\rho_\infty^T(V,\tilde{V})=\sup_{t\in[0,T]}\|V_t-\tilde{V}_t\|_\infty,\ \ V, \tilde{V}\in C([0,T];\scr C^{m+d}).$$
Let $\alpha(\epsilon):=2\e^{(\lambda_1-\epsilon)r} \frac{\|\sigma\|}{\epsilon},\epsilon\in(0,\lambda_1)$. We claim that \cite[Lemma 2.2]{BWY13} holds for $\alpha(\epsilon)$ with some constant $\epsilon\in(0,\lambda_1)$ replacing $\alpha(T)$. To this end, it is sufficient to prove \cite[(14)]{BWY13} for $\alpha(\epsilon)$ with some constant $\epsilon\in(0,\lambda_1)$ instead of $\alpha(T)$, i.e.
\begin{align}\label{X-Y}\sup_{s\in[0,t]}\|X(s)-Y(s)\|_\infty^2\leq \e^{(\lambda_1-\epsilon)r} \frac{\|\sigma\|}{\epsilon}\int_0^th(s)^2\d s, \ \ t\geq 0.
\end{align}
In fact, it follows from It\^{o}'s formula and {\bf(C)} that
$$\d |X(t)-Y(t)|^2\leq \frac{\|\sigma\|}{\epsilon}h(t)^2\d t+(\epsilon-\lambda_1)|X(t)-Y(t)|^2\d t+\lambda_2\|X_t-Y_t\|_\infty^2\d t,\ \ \epsilon\in(0,\lambda_1).$$
So, we get
$$\d [\e^{(\lambda_1-\epsilon)t}|X(t)-Y(t)|^2]\leq \e^{(\lambda_1-\epsilon)t}\frac{\|\sigma\|}{\epsilon}h(t)^2\d t+\e^{(\lambda_1-\epsilon)t}\lambda_2\|X_t-Y_t\|_\infty^2\d t.$$
Let $\eta_t=\sup_{s\in[0,t]}\e^{(\lambda_1-\epsilon)s}|X(s)-Y(s)|$. It follows from $X_0=Y_0$ that
\begin{align*}\eta_t\leq \int_0^t\e^{(\lambda_1-\epsilon)s}\frac{\|\sigma\|}{\epsilon}h(s)^2\d s+\lambda_2\e^{(\lambda_1-\epsilon)r}\int_0^t\eta_s\d s.
\end{align*}
Gronwall's inequality implies that
\begin{align*}\eta_t&\leq \int_0^t\exp\{\lambda_2\e^{(\lambda_1-\epsilon)r}(t-s)\}\e^{(\lambda_1-\epsilon)s}\frac{\|\sigma\|}{\epsilon}h(s)^2\d s\\
&=\int_0^t\exp\{\lambda_2\e^{(\lambda_1-\epsilon)r}t\}\e^{[(\lambda_1-\epsilon)-\lambda_2\e^{(\lambda_1-\epsilon)r}]s}\frac{\|\sigma\|}{\epsilon}h(s)^2\d s.
\end{align*}
Noting that $\eta_t\geq \e^{(\lambda_1-\epsilon)(t-r)}\|X_t-Y_t\|_\infty$, we arrive at
$$\|X_t-Y_t\|_\infty^2\leq \e^{(\lambda_1-\epsilon)r} \frac{\|\sigma\|}{\epsilon}\int_0^t\e^{-\e^{(\lambda_1-\epsilon)r}[(\lambda_1-\epsilon)\e^{-(\lambda_1-\epsilon)r}-\lambda_2](t-s)}h(s)^2\d s.$$
Since $\lambda_2< \sup_{\delta\in[0,\lambda_1]}\delta\e^{-\delta r}$ and $\delta\to\delta\e^{-\delta r}$ is a continuous function, there exists a constant $\epsilon\in(0,\lambda_1)$ such that $(\lambda_1-\epsilon)\e^{-(\lambda_1-\epsilon)r}-\lambda_2>0$. In the following, we fix this $\epsilon$. We derive
$$\|X_t-Y_t\|_\infty^2\leq \e^{(\lambda_1-\epsilon)r} \frac{\|\sigma\|}{\epsilon}\int_0^th(s)^2\d s,\ \ t\geq 0$$
which gives \eqref{X-Y}. So, \cite[Lemma 2.2]{BWY13} holds for $\alpha(\epsilon)$ replacing $\alpha(T)$.
Therefore, by \cite[(7)]{BWY13} with $c_\mu=0$, the transportation cost inequality for $\Pi_\xi^T$ holds, i.e.
\begin{align}\label{TCP}\W_{2,\rho_\infty^T}(\nu^T,\Pi_\xi^T)^2\leq 2\e^{(\lambda_1-\epsilon)r} \frac{\|\sigma\|}{\epsilon}\Ent(\nu^T|\Pi_\xi^T)
\end{align}
for any probability measure $\nu^T$ on $C([0,T];\C^{m+d})$ with $\nu^T(\sup_{t\in[0,T]}\|v_t\|^2_\infty)<\infty$.

Define the projection mapping $\pi_T:C([0,T];\C^{m+d})\to\C^{m+d}$ as $\pi_T(v)=v_T,v\in C([0,T];\C^{m+d})$. Then by \eqref{TCP} and \cite[Lemma 2.1]{DGW} for $\Phi=\pi_T$, we obtain
$$\W_{2}(\nu,P^{\mu^\ast}_T(\xi,\cdot))^2\leq 2\e^{(\lambda_1-\epsilon)r} \frac{\|\sigma\|}{\epsilon}\Ent(\nu|P^{\mu^\ast}_T(\xi,\cdot)),\ \ \nu\in\scr P_2(\C^{m+d}).$$
Finally, in view of \eqref{ECR} and \cite[Lemma 2.2]{DGW}, we complete the proof.
\end{proof}
\section{Propagation of Chaos}
In this section, we consider path-distribution dependent SHS on $\mathbb{R}^{m+d}$:
\beq\label{E1}\begin{split}
\d X(t)&=\left(
                                  \begin{array}{c}
                                    b(t,X_t,\L_{X_t}) \\
                                    B(t,X_t,\L_{X_t}) \\
                                  \end{array}
                                \right)
\d t+\left(
       \begin{array}{c}
         0_{m\times d} \\
         \sigma(t,X_t,\L_{X_t}) \\
       \end{array}
     \right)
\d W(t), \\
\end{split}\end{equation}
where $W=(W(t))_{t\geq 0}$ is a $d$-dimensional standard Brownian motion with respect to a complete filtration probability space $(\OO, \F, \{\F_{t}\}_{t\ge 0}, \P)$, $b:[0,\infty)\times \C^{m+d}\times \scr P(\C^{m+d})\to \mathbb{R}^{m}$,
$B: [0,\infty)\times \C^{m+d}\times \scr P(\C^{m+d})\to \mathbb{R}^{d}$
and $\sigma: [0,\infty)\times \C^{m+d}\times \scr P(\C^{m+d})\to \mathbb{R}^{d}\otimes\mathbb{R}^{d}$ are measurable. Throughout this section, we fix $T>0$ and consider the solution for \eqref{E1} on time interval $[0,T]$.

Let $X_0$ be an $\F_0$-measurable $\C^{m+d}$-valued random variable,
$N\ge1$ be an integer and $(X_0^i,W^i(t))_{1\le i\le N}$ be i.i.d.\,copies of $(X_0,W(t)).$ Consider the following non-interacting particle system:
\begin{equation}\begin{split}\label{GP0}\d X^i(t)&=\left(
                                     \begin{array}{c}
                                       b(t,X^i_t,\mu_t^i) \\
                                       B(t,X^i_t,\mu_t^i) \\
                                     \end{array}
                                   \right)
\d t+\left(
       \begin{array}{c}
         0_{m\times d} \\
         \sigma(t,X^i_t,\mu_t^i) \\
       \end{array}
     \right)
\d W^i(t), \ \ 1\leq i\leq N,
\end{split}\end{equation}
where  $\mu_t^i:=\mathscr{L}_{X_t^i}$,
and the mean field interacting particle system
\begin{align}\label{GPS00}\d X^{i,N}(t)&=\left(
                                                 \begin{array}{c}
                                                   b(t,X^{i,N}_t,\hat\mu_t^N)\\
                                                   B(t,X^{i,N}_t,\hat\mu_t^N) \\
                                                 \end{array}
                                               \right)
\d t+\left(
                      \begin{array}{c}
                        0_{m\times d} \\
                        \sigma(t,X^{i,N}_t,\hat\mu_t^N) \\
                      \end{array}
                    \right)
\d W^i(t),\ \ X_0^{i,N}=X_0^i,
\end{align}
where  $\hat\mu_t^N$ is the empirical distribution of $X_t^{1,N},\cdots,X_t^{N,N}$, i.e.
\begin{equation*}
 \hat\mu_t^N =\ff{1}{N}\sum_{j=1}^N\dd_{X_t^{j,N}}.
 \end{equation*}

 To obtain the propagation of chaos, we make the following assumptions.
\beg{enumerate}
\item[\bf{(H)}] There exist constants $K>0$ and $\theta\geq 1$ such that the following conditions hold for all $t\in[0, T]$ and $\gamma\in\scr P_\theta(\C^{m+d})$:
\item[(H1)]For any $\xi,\eta\in\C^{m+d}$,
\begin{equation*}\beg{split}
&|b(t,\xi,\gamma)-b(t,\eta,\gamma)|+|B(t,\xi,\gamma)-B(t,\eta,\gamma)|+\|\sigma(t,\xi,\gamma)-\sigma(t,\eta,\gamma)\|\leq K\|\xi-\eta\|_{\infty}.
\end{split}\end{equation*}
\item[(H2)] For any $\xi\in \C^{m+d}$ and $\bar{\gamma},\tilde{\gamma}\in \scr P_\theta(\C^{m+d})$,
\begin{align*}
&|b(t,\xi,\bar{\gamma})-b(t,\xi,\tilde{\gamma})|+ \|\sigma(t,\xi,\bar{\gamma})-\sigma(t,\xi,\tilde{\gamma})\|+|B(t,\xi,\bar{\gamma})-B(t,\xi,\tilde{\gamma})|
\leq K\W_\theta(\bar{\gamma},\tilde{\gamma}),\\
&|b(t,0,\delta_0)|+|B(t,0,\delta_0)|+|\sigma(t,0,\delta_0)|\leq K.\end{align*}
\end{enumerate}
 Under {\bf (H)}, the well-posedness in $\scr P_{\theta}(\C^{m+d})$ for \eqref{E1} holds due to
Remark \ref{Rem} below, which means that $\mu_t^i$ in \eqref{GP0} does not depend on $i$ and we denote $\mu_t=\mu^i_t,t\in[0,T]$. Moreover, by Theorem \ref{wip} below, \eqref{GPS00} is also well-posed.

To prove the propagation of chaos, we need the following lemma, which may be a known result. Since we have not found some references, we give a brief proof in the following.
\begin{lem}\label{INT} Let $\{Z_i\}_{i\geq 1}$ be a sequence of i.i.d. non-negative random variables with $\E (Z_1)<\infty$. Then $\{\frac{1}{N}\sum_{i=1}^N Z_i\}_{N\geq 1}$ is uniformly integrable.
\end{lem}
\begin{proof}
Since $\E(Z_1)<\infty$, it follows from the strong law of large number that $\P$-a.s.
$$\lim_{N\to\infty}\frac{1}{N}\sum_{i=1}^N Z_i=\E(Z_1),$$
which yields $\P$-a.s.
$$\sup_{N\geq 1}\left\{\frac{1}{N}\sum_{i=1}^NZ_i\right\}<\infty.$$
This together with the fact that $\{Z_i\}_{i\geq 1}$ are i.i.d., $\E(Z_1)<\infty$ and the dominated convergence theorem yields that
\begin{align*}
&\lim_{M\to\infty}\sup_{N\geq 1}\E\left\{\left(\frac{1}{N}\sum_{i=1}^N Z_i\right)1_{\left\{\frac{1}{N}\sum_{i=1}^N Z_i\geq M\right\}}\right\}\\
&=\lim_{M\to\infty}\sup_{N\geq 1}\E\left\{ Z_11_{\left\{\frac{1}{N}\sum_{i=1}^N Z_i\geq M\right\}}\right\}\\
&\leq \lim_{M\to\infty}\E\left\{ Z_11_{\left\{\sup_{N\geq 1}\{\frac{1}{N}\sum_{i=1}^N Z_i\}\geq M\right\}}\right\}=0.
\end{align*}
So, we complete the proof.
\end{proof}

To derive the quantitative propagation of chaos, we introduce the projection mappings
\begin{align*}\pi(s)(\xi)=\xi(s), \ \ s\in[-r,0],\xi\in\C^{m+d}
\end{align*}
and define $\mu^s=\mu\circ\pi(s)^{-1}, \mu\in\scr P(\C^{m+d})$. Then for any $\mu\in\scr P(\C^{m+d}), s\in[-r,0]$, $\mu^s$ is a probability measure on $\R^{m+d}$. Let  $\W^0_{\theta}$ be the $L^\theta$-Wasserstein distance on $\scr P_\theta(\R^{m+d})$, the collection of all probability measures with finite $\theta$-th moment on $\R^{m+d}$. Let $\Gamma$ be a probability measure on $[-r,0]$ and define
 \begin{align}\label{wtg}
 \W^\Gamma_\theta(\gamma,\bar{\gamma}):=\int_{-r}^{0}\W^0_\theta(\gamma^s,\bar{\gamma}^s)\Gamma(\d s),\ \ \gamma,\bar{\gamma}\in \scr P_\theta(\C^{m+d}).
 \end{align}
 Noting that for any $\gamma,\bar{\gamma}\in \scr P_\theta(\C^{m+d})$, it holds
\begin{align*}|\W^0_\theta(\gamma^t,\bar{\gamma}^t)-\W^0_\theta(\gamma^s,\bar{\gamma}^s)|&\leq |\W^0_\theta(\gamma^t,\bar{\gamma}^t)-\W^0_\theta(\gamma^t,\bar{\gamma}^s)| +|\W^0_\theta(\gamma^t,\bar{\gamma}^s)-\W^0_\theta(\gamma^s,\bar{\gamma}^s)|\\
&\leq \W^0_\theta(\bar{\gamma}^t,\bar{\gamma}^s)+ \W^0_\theta(\gamma^t,\gamma^s),\ \ s,t\in[-r,0].
 \end{align*}
 So, $\W^0_\theta(\gamma^s,\bar{\gamma}^s)$ is continuous in $s$ and the right hand side of \eqref{wtg} is well-defined. Moreover, it is clear that
\begin{align}\label{stc}
\W^\Gamma_\theta(\gamma_1,\gamma_2)\leq \W_\theta(\gamma_1,\gamma_2),\ \  \gamma_1,\gamma_2\in \scr P_\theta(\C^{m+d}).
\end{align}
In particular, when $\Gamma=\delta_0$, $\W^\Gamma_\theta(\gamma,\bar{\gamma})=\W^0_\theta(\gamma^0,\bar{\gamma}^0)$. The main result in this section is as follows.
\begin{thm}\label{POC} Assume {\bf(H)} and $\E\|X_0^i\|_\infty^\theta<\infty$.
Then the following assertions hold.
\begin{enumerate}
\item[(1)] It holds \begin{align}\label{COT00}\lim_{N\to\infty}\E\sup_{t\in[0,T]}|X^i(t)- X^{i,N}(t)|^\theta=0.\end{align}
Consequently,
\begin{align}\label{emp}\lim_{N\to\infty}\E\sup_{t\in[0,T]}\W_\theta(\hat{\mu}_t^N,\mu_t)^\theta=0.
\end{align}
If in addition, $b(t,\xi,\gamma)$ and $\sigma(t,\xi,\gamma)$ do not depend on $\gamma$ and there exists a constant $K_0>0$ such that
\begin{align}\label{sig} \nonumber &|B(t,\xi,\gamma)-B(t,\xi,\tilde{\gamma})|\leq K_0 [\W_\theta(\gamma,\tilde{\gamma})\wedge 1],\\
&\|\sigma(t,\xi)^{-1}\|<K_0, \ \ (t,\xi)\in[0,T]\times\C^{m+d}, \gamma,\tilde{\gamma}\in\scr P_\theta(\C^{m+d}),
\end{align}
then for any $k\geq 1$,
\begin{align}\label{enp}&\nonumber\lim_{N\to\infty}\sup_{t\in[0,T]}\|\L_{(X_t^{1,N},X_t^{2,N},\cdots,X_t^{k,N})}-\mu_t^{\otimes k}\|^2_{var}\\
&\qquad\qquad\qquad\leq 2\lim_{N\to\infty}\sup_{t\in[0,T]}\Ent\left(\L_{(X_t^{1,N},X_t^{2,N},\cdots,X_t^{k,N})}|\mu_t^{\otimes k}\right)=0,\end{align}
where $\mu_t^{\otimes k}=\prod_{i=1}^k\mu_t$, the $k$-independent product of $\mu_t$.
\item[(2)] Assume that $\theta\geq 2$ or $\theta\in[1,2)$ but $\sigma(t,\xi,\gamma)$ does not depend on $\gamma$, $\E\|X_0^i\|_\infty^{q}<\infty$ for some $q>\theta$ and there exists a probability measure $\Gamma$ on $[-r,0]$ such that (H2) holds for
$\W^\Gamma_\theta$ replacing $\W_\theta$,
then there exists a constant $C>0$ depending only on $\theta,q,m+d,T$ and $\E\|X_0^i\|_\infty^{q}$ such that
\begin{equation}\begin{split}\label{S1}
&\E\sup_{t\in[0,T]}|X^i(t)-X^{i,N}(t)|^\theta\le C R_{m+d}(N),
\end{split}\end{equation}
where
\begin{equation*}\begin{split}
R_{m+d}(N)=
\begin{cases}
N^{-\ff{1}{2}}+N^{-\frac{q-\theta}{q}},~~~~~~~~~~~~~~~~~~~~\theta>\frac{m+d}{2}, q\neq 2\theta,\\
N^{-\ff{1}{2}}\log (1+N)+N^{-\frac{q-\theta}{q}},~~~~ ~~\theta=\frac{m+d}{2}, q\neq 2\theta,\\
N^{-\ff{\theta}{m+d}}+N^{-\frac{q-\theta}{q}},~~~~~~~~~~~~~~~~\theta\in[1,\frac{m+d}{2}), q\neq \frac{m+d}{m+d-\theta},
\end{cases}
\end{split}\end{equation*}
and consequently
\begin{equation}\label{S30}\begin{split}
&\sup_{t\in[0,T]}\E\W^\Gamma_\theta(\hat{\mu}_t^N,\mu_t)^\theta\le C R_{m+d}(N).
\end{split}\end{equation}
If in addition, $b(t,\xi,\gamma)$ and $\sigma(t,\xi,\gamma)$ do not depend on $\gamma$ and \eqref{sig}  holds for
$\W^\Gamma_\theta$ replacing $\W_\theta$, then there exists a constant $C>0$ depending on $\theta,q,m+d,T$ and $\E\|X_0^i\|_\infty^{q}$ such that for any $k\geq 1$,
\begin{equation}\label{S10}\begin{split}
&\sup_{t\in[0,T]}\|\L_{(X_t^{1,N},X_t^{2,N},\cdots,X_t^{k,N})}-\mu_t^{\otimes k}\|^2_{var}\leq 2\sup_{t\in[0,T]}\Ent\left(\L_{(X_t^{1,N},X_t^{2,N},\cdots,X_t^{k,N})}|\mu_t^{\otimes k}\right)\\
&\le Ck R_{m+d}(N)1_{\{\theta\in[1,2)\}}+CkR_{m+d}(N)^{\frac{2}{\theta}}1_{\{\theta\geq 2\}}.
\end{split}\end{equation}
\end{enumerate}
\end{thm}
\begin{proof} (1) If $\E\|X^i_0\|^p_\infty<\infty$ for some $p\geq \theta$, it is standard to derive from  {\bf(H)} that
\begin{align}\label{mmo}
\E\sup_{t\in[0,T]}\|X^{i}_t\|_\infty^p<C_0(1+\E(\|X^{i}_0\|_\infty^p))
\end{align}
for some constant $C_0>0$.
Let $\eta^{i,N}(t)=\sup_{s\in[0,t]}|X^{i,N}(s)-X^i(s)|$. Applying the BDG inequality and H\"{o}lder's inequality, we derive from {\bf(H)} that
\begin{equation}\label{bbn}
\begin{split}
\E\eta^{i,N}(t)^\theta&\le
c_0\int_{0}^{t}\E(\eta^{i,N}(s)^\theta+\W_\theta(\hat\mu^N _s,\mu_s)^\theta)\d s\\
&+c_0\E\left(\int_{0}^{t}(\eta^{i,N}(s)+\W_\theta(\hat\mu^N _s,\mu_s))^2\d s\right)^{\frac{\theta}{2}}
\end{split}
\end{equation}
for some constant $c_0>0$.
Let $\tt\mu_t^N =\ff{1}{N}\sum_{j=1}^N\dd_{X_t^j}$. Noting that
\begin{align}\label{wwe}\mathbb{W}_\theta(\hat\mu^N_s,\tt\mu^N_s)\leq \left(\frac{1}{N}\sum_{i=1}^N\|X^{i,N}_s-X^i_s\|_\infty^\theta\right)^{\frac{1}{\theta}},
\end{align}
we obtain
\begin{equation}\label{A4}
\begin{split}
\W_\theta(\hat\mu^N _s,\mu_s)&\le
\mathbb{W}_\theta(\hat\mu^N_s,\tt\mu^N_s)+\mathbb{W}_\theta(\tt\mu^N_s,\mu_s)\\
&\le
\left(\frac{1}{N}\sum_{i=1}^N\|X^{i,N}_s-X^i_s\|_\infty^\theta\right)^{\frac{1}{\theta}} +\mathbb{W}_\theta(\tt\mu^N_s,\mu_s).
\end{split}
\end{equation}
Next, we divide into two cases: $\theta\geq 2$ and $\theta\in[1,2)$ to estimate the second term on the right hand side of \eqref{bbn}.

If $\theta\geq 2$, by H\"{o}lder's inequality, we have
\begin{equation*}
\begin{split}
c_0\E\left(\int_0^t(\eta^{i,N}(s)+\W_\theta(\hat\mu^N _s,\mu_s))^2 \d s\right)^{\frac{\theta}{2}}\le
c_1\int_{0}^{t}\E\eta^{i,N}(s)^{\theta}\d s+c_1\E\int_0^t\W_\theta(\hat\mu^N _s,\mu_s)^{\theta}\d s
\end{split}
\end{equation*}
for some constant $c_1>0.$
This together with \eqref{bbn} and \eqref{A4} implies that there exists a constant $c_2>0$ such that
\begin{equation*}
\begin{split}
\E\eta^{i,N}(t)^\theta&\le
c_2\int_{0}^{t}\E\eta^{i,N}(s)^\theta\d s+c_2\E\int_0^t\mathbb{W}_\theta(\tt\mu^N_s,\mu_s)^\theta\d s.
\end{split}
\end{equation*}
\eqref{mmo} for $p=\theta$ and Gronwall's inequality give
\begin{equation}\label{LEW00}
\E\eta^{i,N}(t)^{\theta}\le c_3\E\int_0^t\mathbb{W}_\theta(\tt\mu^N_s,\mu_s)^\theta\d s
\end{equation}
for some constant $c_3>0$.

If $\theta\in[1,2)$, it follows from \eqref{A4}, the inequality $\sqrt{|ab|}\leq \frac{|a|+|b|}{2}$ and H\"{o}lder's inequality that
\begin{equation}\label{tte}
\begin{split}
&c_0\E\left(\int_0^t(\eta^{i,N}(s)+\W_\theta(\hat\mu^N _s,\mu_s))^2 \d s\right)^{\frac{\theta}{2}}\\
&\leq c_0\E\left(\int_0^t\left(\eta^{i,N}(s)+\left(\frac{1}{N}\sum_{i=1}^N\|X^{i,N}_s-X^i_s\|_\infty^\theta\right)^{\frac{1}{\theta}} +\mathbb{W}_\theta(\tt\mu^N_s,\mu_s)\right)^2 \d s\right)^{\frac{\theta}{2}}\\
&\le
c'_1\int_{0}^{t}\E\eta^{i,N}(s)^{\theta}\d s+\frac{1}{2}\E\eta^{i,N}(t)^\theta+c'_1\E\left(\int_0^t\W_\theta(\tilde{\mu}^N _s,\mu_s)^{2}\d s\right)^{\frac{\theta}{2}}
\end{split}
\end{equation}
for some constant $c'_1>0$.
So, this combined with \eqref{bbn} and \eqref{A4} derives
\begin{equation}\label{n00}
\begin{split}
\E\eta^{i,N}(t)^\theta&\le
c'_2\int_{0}^{t}\E\eta^{i,N}(s)^\theta\d s+c'_2\E\int_0^t\mathbb{W}_\theta(\tt\mu^N_s,\mu_s)^\theta\d s\\
&\qquad+c'_2\E\left(\int_0^t\W_\theta(\tilde{\mu}^N _s,\mu_s)^{2}\d s\right)^{\frac{\theta}{2}}
\end{split}
\end{equation}
for some constant $c'_2>0$.
Therefore, using Gr\"{o}nwall's inequality for \eqref{n00}, there exists a constant $c'_3>0$ such that
\begin{equation}\label{LEW}
\E\eta^{i,N}(t)^{\theta}\le c'_3\E\int_0^t\mathbb{W}_\theta(\tt\mu^N_s,\mu_s)^\theta\d s+c'_3\E\left(\int_0^t\W_\theta(\tilde{\mu}^N _s,\mu_s)^{2}\d s\right)^{\frac{\theta}{2}}.
\end{equation}
Let $\C_T^{m+d}=C([-r,T];\R^{m+d})$ be equipped with the uniform norm and $\scr P(\C_T^{m+d})$ be the set of the probability measures on $\C_T^{m+d}$. Define
$$\scr P_\theta(\C_T^{m+d})=\left\{\mu^T\in \scr P(\C_T^{m+d}): \int_{\C_T^{m+d}}\sup_{s\in[-r,T]}|\xi(s)|^\theta\mu^T(\d \xi)<\infty\right\}$$
and denote $\W_{\theta,T}$ as the $L^\theta$-Wasserstein distance on $\scr P_\theta(\C_T^{m+d})$. So, $(\scr P_\theta(\C_T^{m+d}),\W_{\theta,T})$ is a Polish space.

Next, by the triangle inequality, we arrive at
\begin{align}\label{sup00}\nonumber \sup_{s\in[0,T]}\mathbb{W}_\theta(\tt\mu^N_s,\mu_s)&\leq \W_{\theta,T}\left(\frac{1}{N}\sum_{i=1}^N\delta_{X^{i}([-r,T])},\L_{X^i([-r,T])}\right) \\
&\leq \W_{\theta,T}\left(\frac{1}{N}\sum_{i=1}^N\delta_{X^{i}([-r,T])},\delta_0\right) +\W_{\theta,T}\left(\delta_0,\L_{X^i([-r,T])}\right) \\
\nonumber&\leq \left(\frac{1}{N}\sum_{i=1}^N\sup_{s\in[0,T]}\|X^{i}_s\|_\infty^{\theta}\right)^{\frac{1}{\theta}} +\E\left(\sup_{s\in[0,T]}\|X^{i}_s\|_\infty^{\theta}\right)^{\frac{1}{\theta}}.
\end{align}
Thanks to the generalized Glivenko-Cantelli-Varadarajan theorem, see for instance \cite[Corollary 12.2.2]{RKF}, it holds $\P$-a.s.
\begin{align}\label{LLN}\lim_{N\to\infty}\W_{\theta,T}\left(\frac{1}{N}\sum_{i=1}^N\delta_{X^{i}([-r,T])}, \L_{X^i([-r,T])}\right)=0.
 \end{align}
Therefore, it follows from \eqref{mmo} for $p=\theta$, \eqref{sup00}, \eqref{LLN}, Lemma \ref{INT} for $Z_i=\sup_{s\in[0,T]}\|X^{i}_s\|_\infty^{\theta}$ and the dominated convergence theorem that
\begin{align}\label{lnt}\lim_{N\to\infty}\E\sup_{s\in[0,T]}\mathbb{W}_\theta(\tt\mu^N_s,\mu_s)^\theta\leq \lim_{N\to\infty} \E\left[\W_{\theta,T}\left(\frac{1}{N}\sum_{i=1}^N\delta_{X^{i}([-r,T])},\L_{X^i([-r,T])}\right)^\theta\right]=0.
\end{align}
This and \eqref{LEW00} or \eqref{LEW} derive \eqref{COT00}. Finally,
by \eqref{COT00}, \eqref{lnt} and \eqref{A4}, we get \eqref{emp}.

When $b(t,\xi,\gamma)$ and $\sigma(t,\xi,\gamma)$ do not depend on $\gamma$, we can rewrite \eqref{GP0} as
\begin{equation*}\begin{split}\label{GPS}\d X^i(t)&=\left(
                                     \begin{array}{c}
                                       b(t,X^i_t) \\
                                       B(t,X^i_t,\frac{1}{N}\sum_{i=1}^N\delta_{X_t^i}) \\
                                     \end{array}
                                   \right)
\d t+\left(
       \begin{array}{c}
         0_{m\times d} \\
         \sigma(t,X^i_t) \\
       \end{array}
     \right)
\d \tilde{W}^i(t), \ \ 1\leq i\leq N,
\end{split}\end{equation*}
with
$$\d \tilde{W}^i(t)=\d W^i(t)-\tilde{\Gamma}^i(t)\d t,\ \ 1\leq i\leq N$$
and $$\tilde{\Gamma}^i(t)=\sigma(t,X^i_t)^{-1}[B(t,X^{i}_t, \frac{1}{N}\sum_{i=1}^N\delta_{X_t^i})-B(t,X^{i}_t, \mu_t)],\ \ 1\leq i\leq N.$$
It follows from \eqref{sig} that
\begin{align}\label{gam}|\tilde{\Gamma}^i(t)|\leq K_0^2(\W_{\theta}(\frac{1}{N}\sum_{i=1}^N\delta_{X_t^i},\mu_t)\wedge 1),\ \ t\in[0,T],1\leq i\leq N.\end{align}
Let $$R_t=\exp{\left\{\sum_{i=1}^N\int_0^t\<\tilde{\Gamma}^i(s),\d W^i(s)\>-\frac{1}{2}\sum_{i=1}^N\int_0^t|\tilde{\Gamma}^i(s)|^2\d s \right\}},\ \ t\in[0,T].$$ \eqref{gam} and Girsanov's theorem imply that $\{R_t\}_{t\in[0,T]}$ is a martingale and $((\tilde{W}^i(t))_{1\leq i\leq N})_{t\in[0,T]}$ is an $Nd$-dimensional Brownian motion under $\Q_T=R_T \P$ and
$$\L_{(X^{1}_t,X^{2}_t,\cdots,X^{N}_t)}|\Q_T=\L_{(X^{1,N}_t,X^{2,N}_t,\cdots,X^{N,N}_t)}|\P,\ \ t\in[0,T].$$
This implies that
\begin{align*}\E [f(X^{1,N}_t,X^{2,N}_t,\cdots,X^{N,N}_t)]&=\E [R_T f(X^{1}_t,X^{2}_t,\cdots,X^{N}_t)]\\
&=\E [R_t f(X^{1}_t,X^{2}_t,\cdots,X^{N}_t)],\ \ f\in\scr B_b(\C^{N(m+d) }), t\in[0,T].
\end{align*}
So,  there exists  a constant $C>0$ such that
\begin{align*}&\Ent(\L_{(X^{1,N}_t,X^{2,N}_t,\cdots,X^{N,N}_t)}|\P|\mu_t^{\otimes N})\\
&\leq \E(R_t\log R_t)= \frac{1}{2}\sum_{i=1}^N\int_0^t\E^{\Q_T}|\tilde{\Gamma}^i(s)|^2\d s\\
&\leq  C^2N\int_0^t\E^{\Q_T}(\W_{\theta}(\frac{1}{N}\sum_{i=1}^N\delta_{X_s^i},\mu_s)\wedge 1)^2\d s\\
&=C^2N\int_0^t\E(\W_{\theta}(\frac{1}{N}\sum_{i=1}^N\delta_{X_s^{i,N}},\mu_s)\wedge 1)^2\d s\\
&=C^2N\int_0^t\E(\W_{\theta}(\hat{\mu}^N_s,\mu_s)\wedge 1)^2\d s,\ \ t\in[0,T].\end{align*}
This together with \cite[Lemma 3.9]{MD} implies that for any $k\geq 1$ and $N\geq k$,
\beg{align*} &\mathrm{Ent}(\L_{(X_{t}^{1,N},X_{t}^{2,N},\cdots, X_{t}^{k,N})}|\mu_t^{\otimes k})\leq 2C^2k\int_0^t\E(\W_{\theta}(\hat{\mu}^N_s,\mu_s)\wedge 1)^2\d s.\end{align*}
So, Pinsker's inequality \eqref{ETX} yields
\begin{align}\label{aae}\nonumber\|\L_{(X^{1,N}_t,X^{2,N}_t,\cdots,X^{k,N}_t)}-\mu_t^{\otimes k}\|_{var}^2&\leq 2\Ent(\L_{(X^{1,N}_t,X^{2,N}_t,\cdots,X^{k,N}_t)}|\mu_t^{\otimes k})\\
&\leq 4C^2k\int_0^t\E(\W_{\theta}(\hat\mu_s^N,\mu_s)\wedge 1)^2\d s.
\end{align}
Note that
\begin{align}\label{wth}\E(\W_{\theta}(\hat\mu_s^N,\mu_s)\wedge 1)^2\leq \E(\W_{\theta}(\hat\mu_s^N,\mu_s)^\theta)1_{\{\theta\in[1,2)\}}+ \left(\E\W_{\theta}(\hat\mu_s^N,\mu_s)^\theta\right)^{\frac{2}{\theta}}1_{\{\theta\geq 2\}}.
\end{align}
By \eqref{emp} and \eqref{aae}, we prove \eqref{enp}.

(2) Assume that (H2) holds for
$\W^\Gamma_\theta$ replacing $\W_\theta$. When $\theta\geq 2$, repeating the proof to get \eqref{LEW00}, we derive
\begin{equation}\label{EEL}
\E\eta^{i,N}(t)^{\theta}\le c_4\int_0^t\E\mathbb{W}^\Gamma_\theta(\tt\mu^N_s,\mu_s)^\theta\d s
\end{equation}
for some constant $c_4>0$.
When $\theta\in[1,2)$ but $\sigma(t,\xi,\gamma)$ does not depend on $\gamma$, \eqref{tte} is replaced by
\begin{equation*}
\begin{split}
c_0\E\left(\int_0^t\eta^{i,N}(s)^2 \d s\right)^{\frac{\theta}{2}}
&\le
c_4'\int_{0}^{t}\E\eta^{i,N}(s)^{\theta}\d s+\frac{1}{2}\E\eta^{i,N}(t)^\theta
\end{split}
\end{equation*}
for some constant $c'_4>0$.
Then \eqref{EEL} instead of \eqref{LEW} holds.
Next, by the definition of $\W_\theta^\Gamma$, we have
\begin{equation}\label{EEL00}
\E\mathbb{W}^\Gamma_\theta(\tt\mu^N_s,\mu_s)^\theta\leq \int_{-r}^0\E\left[\W_\theta^0\left(\ff{1}{N}\sum_{i=1}^N\dd_{X_s^i(u)},\L_{X_s^i(u)}\right)^\theta\right]\Gamma(\d u).
\end{equation}
Note that $\sup_{t\in[0,T]}\mu_t(\|\cdot\|_\infty^q)<\infty$ due to \eqref{mmo} for $p=q$. By \cite[Theorem 1]{FG} for $p=\theta,q=q$, there exists a constant $C_0>0$ depending only on $\theta,q,m+d$ such that
\begin{align*}&\E\left[\W_\theta^0\left(\ff{1}{N}\sum_{i=1}^N\dd_{X_s^i(u)},\L_{X_s^i(u)}\right)^\theta\right]\\
&\le C_0\left(\sup_{t\in[0,T]}\mu_t(\|\cdot\|_\infty^q)\right)^{\frac{\theta}{q}}
R_{m+d}(N),\ \ s\in[0,T],u\in[-r,0].
\end{align*}
Substituting this into \eqref{EEL00}, we derive from \eqref{mmo} for $p=q$ that
\begin{equation}\label{EEL0n}
\sup_{s\in[0,T]}\E\mathbb{W}^\Gamma_\theta(\tt\mu^N_s,\mu_s)^\theta\le C_0 \left(\sup_{t\in[0,T]}\mu_t(\|\cdot\|_\infty^q)\right)^{\frac{\theta}{q}}
R_{m+d}(N)\leq CR_{m+d}(N).
\end{equation}
So, \eqref{S1} follows from \eqref{EEL} and \eqref{EEL0n}. Moreover, it follows from \eqref{stc} and \eqref{wwe} that
\begin{equation*}
\begin{split}
\W_\theta^\Gamma(\hat\mu^N _s,\mu_s)^\theta&\le
2^{\theta-1}\mathbb{W}^\Gamma_\theta(\hat\mu^N_s,\tt\mu^N_s)^\theta +2^{\theta-1}\mathbb{W}^\Gamma_\theta(\tt\mu^N_s,\mu_s)^\theta\\
&\le
2^{\theta-1}\mathbb{W}_\theta(\hat\mu^N_s,\tt\mu^N_s)^\theta +2^{\theta-1}\mathbb{W}^\Gamma_\theta(\tt\mu^N_s,\mu_s)^\theta\\
&\le2^{\theta-1}
\frac{1}{N}\sum_{i=1}^N\|X^{i,N}_s-X^i_s\|_\infty^\theta +2^{\theta-1}\mathbb{W}^\Gamma_\theta(\tt\mu^N_s,\mu_s)^\theta,
\end{split}
\end{equation*}
which implies \eqref{S30} due to \eqref{S1} and \eqref{EEL0n}.

Finally, if $b(t,\xi,\gamma)$ and $\sigma(t,\xi,\gamma)$ do not depend on $\gamma$ and \eqref{sig} holds for
$\W^\Gamma_\theta$ replacing $\W_\theta$, then \eqref{aae} holds for $\W^\Gamma_\theta$ replacing $\W_\theta$. Moreover, by \eqref{wth} for $\W^\Gamma_\theta$ replacing $\W_\theta$ and \eqref{S30}, we derive \eqref{S10} and the proof is completed.
\end{proof}

\section{Appendix}
In this section, we give the well-posedness of general path-distribution dependent SDEs as well as mean field interacting particle system, and then apply it to the path-distribution dependent SHS.
Fix $T>0$. Let $n,k\in \mathbb{N}^+$ and $\theta\geq 1$. Consider path-distribution dependent SDEs on $\R^n$:
\begin{align}\label{PDD}
\d X(t)=H(t,X_t,\L_{X_t})\d t+\Sigma(t,X_t,\L_{X_t})\d W(t),\ \ t\in[0,T].
\end{align}
where $H:[0,T]\times\scr C^n\times \scr P(\C^n)\to \mathbb{R}^n$, $\Sigma:[0,T]\times\C^n\times \scr P(\C^n)\to\mathbb{R}^{n}\otimes\mathbb{R}^{k}$ are measurable and $W(t)$ is a $k$-dimensional Brownian motion on some complete filtration probability space $(\Omega, \scr F, (\scr F_t)_{t\geq 0},\P)$.
Let $\hat{\scr P}(\C^n)$ be a subset of $\scr P(\C^n)$ and it is equipped with some topology.
\begin{defn}\label{defws} The SDE \eqref{PDD} is called well-posed for distributions in $\hat{\scr P}(\C^n)$, if for any $\F_0$-measurable initial value
$X_0$ with $\L_{X_0}\in \hat{\scr P}(\C^n)$ (respectively any initial distribution $\gamma\in\hat{\scr P}(\C^n)$), it has a unique strong solution (respectively weak solution) such that $\L_{X_\cdot}\in C([0,T];\hat{\scr P}(\C^n))$, the space of continuous maps from $[0,T]$ to $\hat{\scr P}(\C^n)$. In particular, \eqref{PDD} is called well-posed for distributions in $\scr P_\theta(\C^n)$, if the above holds for $(\scr P_\theta(\C^n),\W_\theta)$ replacing $\hat{\scr P}(\C^n)$.
\end{defn}

\begin{thm}\label{GWP} Assume that there exists some constant $K\geq0$ such that
\begin{align}\label{HTL}
\nonumber&|H(s,\xi,\gamma_1)-H(s,\eta,\gamma_2)| + |\Sigma(s,\xi,\gamma_1)-\Sigma(s,\eta,\gamma_2)|\leq K(\|\xi-\eta\|_\infty+\W_\theta(\gamma_1,\gamma_2)),\\
&|H(s,0,\delta_0)| + |\Sigma(s,0,\delta_0)|\leq K,\ \ s\in[0,T], \xi,\eta\in\C^n, \gamma_1,\gamma_2\in \scr P_{\theta}(\C^n).
\end{align}
Then \eqref{PDD} is strongly well-posed in $\scr P_{\theta}(\C^n)$ and there exists a constant $C>0$ such that
\begin{align*}\W_\theta(P_t^\ast \mu_0,P_t^\ast \nu_0)\leq C\e ^{Ct}\W_\theta(\mu_0,\nu_0),\ \ t\in[0,T],\mu_0,\nu_0\in\scr P_\theta(\C^n),
\end{align*}
here $P_t^\ast \mu_0$ is the distribution of the solution to \eqref{PDD} with initial distribution $\mu_0\in\scr P_\theta(\C^n)$.
\end{thm}
\begin{proof}
It follows from \eqref{HTL} that for any $\mu\in C([0,T],\scr P_\theta(\C^n))$, the classical SDE
\begin{align}\label{PDP}
\d X^\mu(t)=H(t,X^\mu_t,\mu_t)\d t+\Sigma(t,X_t^\mu,\mu_t)\d W(t),\ \ t\in[0,T]
\end{align}
is well-posed. For any $\F_0$-measurable random variable $X_0$ with $\L_{X_0}\in\scr P_\theta(\C^n)$, let $X_t^{\mu,X_0}$ be the unique solution to \eqref{PDP} starting from $X_0$. Define the mapping $\Phi^{X_0}:C([0,T],\scr P_\theta(\C^n))\to C([0,T],\scr P_\theta(\C^n))$ as
$$\Phi_t^{X_0}(\mu)=\L_{X_t^{\mu,X_0}},\ \ t\in[0,T].$$
By \eqref{HTL} and the inequality $$(|a|+|b|+|c|)^{\theta}\leq 3^{\theta-1}(|a|^\theta+|b|^\theta+|c|^\theta),$$ we arrive at
\begin{align}\label{mnu00}
\nonumber|X^{\nu,\tilde{X}_0}(t)-X^{\mu,X_0}(t)|^\theta&\leq 3^{\theta-1}|\tilde{X}(0)-X(0)|^\theta \\ &+3^{\theta-1}\left|\int_0^t[H(s,X^{\nu,\tilde{X}_0}_s,\nu_s)-H(s,X^{\mu,X_0}_s,\mu_s)]\d s\right|^\theta\\
\nonumber&+3^{\theta-1}\left|\int_0^t[\Sigma(s,X^{\nu,\tilde{X}_0}_s,\nu_s) -\Sigma(s,X^{\mu,X_0}_s,\mu_s)]\d W(s)\right|^\theta.
\end{align}
Let $\xi_t=\sup_{s\in[-r,t]}|X^{\mu,X_0}(s)-X^{\nu,\tilde{X}_0}(s)|$.
By \eqref{HTL}, it follows from BDG's inequality, the inequality $\sqrt{|ab|}\leq \frac{|a|+|b|}{2}$ and H\"{o}lder's inequality that
\begin{align*}
&3^{\theta-1}\E\sup_{v\in[0,t]}\left|\int_{0}^{v}\{\Sigma(s,X^{\mu,X_0}_s,\mu_s) -\Sigma(s,X^{\nu,\tilde{X}_0}_s,\nu_s)\}\d W(s)\right|^{\theta}\\
&\leq C_0\E\left(\int_0^t(\xi_s^2+\W_\theta(\mu_s,\nu_s)^2)\d s\right)^{\frac{\theta}{2}}\\
&\leq \frac{1}{2}\E\xi_t^\theta ++C_1\int_0^t\E\xi_s^{\theta}\d s+C_1\left(\int_0^t\W_\theta(\mu_s,\nu_s)^2\d s\right)^{\frac{\theta}{2}}
\end{align*}
for some constant $C_1>0$.
Again by \eqref{HTL} and H\"{o}lder's inequality, there exists a constant $C_2>0$ such that
\begin{align*}
&3^{\theta-1}\E\sup_{v\in[0,t]}\left|\int_0^v[H(s,X^{\mu,X_0}_s,\mu_s) -H(s,X^{\nu,\tilde{X}_0}_s,\nu_s)]\d s\right|^{\theta}\leq C_2\E\int_0^t(\xi_s^\theta+\W_\theta(\mu_s,\nu_s)^\theta)\d s.
\end{align*}
As a result, we obtain from \eqref{mnu00} and H\"{o}lder's inequality that
\begin{align*}
\E\xi_t^{\theta}&\leq 2^{\theta-1}\E\|X_0-\tilde{X}_0\|_\infty^\theta+ 2^{\theta-1}\E\sup_{s\in[0,t]}|X^{\mu,X_0}(s)-X^{\nu,\tilde{X}_0}(s)|^{\theta}\\
&\leq C_3\E\|X_0-\tilde{X}_0\|_\infty^\theta+C_3\int_0^t\E\xi_s^{\theta}\d s+C_3\int_0^t\W_\theta(\mu_s,\nu_s)^{\theta}\d s+C_3\left(\int_0^t\W_\theta(\mu_s,\nu_s)^2\d s\right)^{\frac{\theta}{2}}
\end{align*}
for some constant $C_3>0$. So, Gronwall's inequality yields that there exists a constant $C_4>0$ such that
 \begin{align}\label{uuu} \nonumber\W_{\theta}(\Phi^{X_0}_t(\mu),\Phi^{\tilde{X}_0}_t(\nu))^\theta\leq \E\xi_t^{\theta}&\leq
C_4\E\|X_0-\tilde{X}_0\|_\infty^\theta+C_4\int_0^t\W_{\theta}(\mu_s, \nu_s)^{\theta}\d s\\
&\qquad\qquad\qquad+C_4\left(\int_0^t\W_\theta(\mu_s,\nu_s)^2\d s\right)^{\frac{\theta}{2}},\ \  t\in[0,T].
\end{align}
Therefore, for any $\delta>0$, we have
  \begin{align*}
\sup_{t\in[0,T]}\e^{-\delta \theta t}\W_{\theta}(\Phi^{X_0}_t(\mu),\Phi^{X_0}_t(\nu))^\theta\leq \sup_{t\in[0,T]}\e^{-\delta\theta t}\W_{\theta}(\mu_t, \nu_t)^\theta C_4[(\delta\theta)^{-1}+(2\delta)^{-\frac{\theta}{2}}].
\end{align*}
Take $\delta_0$ satisfying $\left(C_4[(\delta_0\theta)^{-1}+(2\delta_0)^{-\frac{\theta}{2}}]\right)^{\frac{1}{\theta}}<\frac{1}{2}$ and
let
$E^{X_0}:= \{\mu\in C([0,T];\scr P_\theta(\C^n)):\mu_0=\L_{X_0}\}$ equipped with the complete metric
$$\rr_{\delta_0}(\nu,\mu):= \sup_{t\in [0,T]}\e^{-\delta_0t} \W_{\theta}(\nu_t,\mu_t),\ \ \mu,\nu\in E^{X_0}.$$
Then we conclude that
  \begin{align*}
\rho_{\delta_0}(\Phi^{X_0}_t(\mu),\Phi^{X_0}_t(\nu))< \frac{1}{2}\rho_{\delta_0}(\mu_t,\nu_t),\ \ \mu,\nu\in E^{X_0},
\end{align*}
and the Banach fixed point theorem yields that
\begin{equation*} \Phi^{X_0}_t(\mu)= \mu_t,\ \ t\in [0,T]\end{equation*}  has a unique solution $\mu\in E^{X_0}.$ This means that \eqref{PDD} has a unique strong solution on $[0,T]$ with initial value $X_0$.

Next, applying \eqref{uuu} for $\mu_t=\L_{X_t}, \nu_t=\L_{\tilde{X}_t}$ and noting
that
\begin{align*}C_4\left(\int_0^t\W_\theta(\mu_s,\nu_s)^2\d s\right)^{\frac{\theta}{2}}&\leq \frac{1}{2}\sup_{s\in[0,t]}\W_{\theta}(\L_{X_s},\L_{\tilde{X}_s})^\theta+C_5\left(\int_0^t\W_\theta(\L_{X_s},\L_{\tilde{X}_s})\d s\right)^{\theta}\\
&\leq \frac{1}{2}\sup_{s\in[0,t]}\W_{\theta}(\L_{X_s},\L_{\tilde{X}_s})^\theta+ C_6\int_0^t\W_\theta(\L_{X_s},\L_{\tilde{X}_s})^{\theta}\d s
\end{align*}
for some constant $C_6>0$, there exists a constant $C_7>0$ such that
\begin{align*}
&\sup_{s\in[0,t]}\W_{\theta}(\L_{X_s},\L_{\tilde{X}_s})^\theta \leq C_7\E\|X_0-\tilde{X}_0\|_\infty^\theta+
C_7\int_0^t\W_{\theta}(\L_{X_s},\L_{\tilde{X}_s})^\theta\d s,\ \ t\in[0,T].
\end{align*}
So, by Gr\"{o}nwall inequality and taking infimum for all $X_0,\tilde{X}_0$ satisfying $\L_{X_0}=\mu_0,\L_{\tilde{X}_0}=\nu_0$, we complete the proof.
\end{proof}
\begin{rem}\label{Rem}
Under {\bf(A2)}-{\bf(A3)}, the assertions in Theorem \ref{GWP} hold for \eqref{EH} replacing \eqref{PDD} by applying Theorem \ref{GWP} for $n=m+d$, $k=d$ and
$$H(t,\xi,\gamma)=\left(
        \begin{array}{c}
          A\xi^{(1)}(0)+M\xi^{(2)}(0) \\
          Z(\xi(0),\gamma)+B(\xi,\gamma)\\
        \end{array}
      \right),
\ \ \Sigma=\left(
                                                \begin{array}{c}
                                                  0_{m\times d} \\
                                                  \sigma \\
                                                \end{array}
                                              \right).
$$
Similarly, under {\bf (H)}, the assertions in Theorem \ref{GWP} hold for \eqref{E1} replacing \eqref{PDD}.
\end{rem}

Next, consider the mean field interacting particle system:
\begin{align}\label{HSM}\d X^{i,N}(t)=      H(t,X^{i,N}_t,\hat\mu_t^N)+\Sigma(t,X^{i,N}_t,\hat\mu_t^N) \d W^i(t),\ \ 1\leq i\leq N,
\end{align}
with $\hat\mu_t^N=\frac{1}{N}\sum_{i=1}^N\delta_{X_t^{i,N}}$ and $(W^i)_{1\leq i\leq N}$ are independent $k$-dimensional standard Brownian motions.
We give a result on the well-posedness of \eqref{HSM}.
\begin{thm}\label{wip}
Under \eqref{HTL}, \eqref{HSM} is well-posed.
\end{thm}
\begin{proof} For any $\xi=\left(
                           \begin{array}{c}
                             \xi_1 \\
                             \xi_2 \\
                             \cdots \\
                             \xi_{N}\\
                           \end{array}
                         \right)
\in(\C^{n})^{N}$, let $\mu_N^{\xi}=\frac{1}{N}\sum_{i=1}^N\delta_{\xi_i}$ and define
$$\tilde{H}(t,\xi)=\left(
                                   \begin{array}{c}
                                     H(t,\xi_1,\mu_N^\xi) \\
                                     H(t,\xi_2, \mu_N^\xi) \\
                                     \cdots \\
                                     H(t,\xi_N,\mu_N^\xi) \\
                                   \end{array}
                                 \right),
\tilde{\Sigma}(t,\xi)=\left(
                                   \begin{array}{cccc}
                                     \Sigma(t,\xi_1,\mu_N^\xi) & 0_{n\times k} & \cdots & 0_{n\times k} \\
                                     0_{n\times k} & \Sigma(t,\xi_2,\mu_N^\xi) &\cdots & 0_{n\times k} \\
                                     \cdots & \cdots & \cdots& \cdots \\
                                     0_{n\times k} & 0_{n\times k} & \cdots & \Sigma(t,\xi_N,\mu_N^\xi) \\
                                   \end{array}
                                 \right).
$$
Note that for $\xi,\eta\in(\C^{n})^{N}$, it holds
\begin{align}\label{wdi00}\nonumber\W_{\theta}(\mu_N^\xi,\mu_N^\eta)&=\W_{\theta}\left(\frac{1}{N}\sum_{i=1}^N\delta_{\xi_i}, \frac{1}{N}\sum_{i=1}^N\delta_{\eta_i}\right)\\
&\leq \left(\frac{1}{N}\sum_{i=1}^N\|\xi_i-\eta_i\|_\infty^\theta\right)^{\frac{1}{\theta}}\leq c(\theta,N)\|\xi-\eta\|_\infty
\end{align}
for some constant $c(\theta,N)>0$.
Consider path dependent SDE on $\R^{nN}$:
\begin{align}\label{NDE}
\d X(t)&=\tilde{H}(t,X_t)\d t+\tilde{\Sigma}(t,X_t)\d W_N(t),
\end{align}
where $W_N=\left(
             \begin{array}{c}
               W^1 \\
               W^2 \\
               \cdots \\
               W^N \\
             \end{array}
           \right)
$ is a $k N$-dimensional Brownian motion.
By \eqref{HTL} and \eqref{wdi00},  we have
\begin{equation*}\beg{split}
&|\tilde{H}(t,\xi)-\tilde{H}(t,\eta)|+\|\tilde{\Sigma}(t,\xi)-\tilde{\Sigma}(t,\eta)\|\leq C\|\xi-\eta\|_{\infty},\ \ \xi,\eta\in(\C^{n})^{N}.
\end{split}\end{equation*}
So, it is standard that \eqref{NDE} is well-posed and so is \eqref{HSM}.
\end{proof}

\beg{thebibliography}{99}

\bibitem{AZ} G. B. Arous, O. Zeitouni, \emph{Increasing propagation of chaos for mean field models,} Ann. Inst. Henri Poincar\'e Probab. Stat. 35(1999), 85-102.


\bibitem{BR1} V. Barbu, M. R\"ockner, \emph{Probabilistic representation for solutions to non-linear Fokker-Planck equations,} SIAM J. Math. Anal. 50(2018), 4246-4260.

\bibitem{BR2} V. Barbu and M. R\"ockner, \emph{From non-linear Fokker-Planck equations to solutions of distribution dependent SDE,} Ann. Probab. 48(2020), 1902-1920.

\bibitem{BWY13} J. Bao, F.-Y. Wang, C. Yuan, \emph{Transportation Cost Inequalities for Neutral Functional Stochastic Equations,} Z. Anal. Anwend. 32(2013), 457-475.


\bibitem{BWY} J. Bao, F.-Y. Wang, C. Yuan, \emph{Derivative formula and Harnack inequality for degenerate functionals SDEs,} Stoch. Dyn. 13(2013), 943-951.



\bibitem{BO} R. J. Berman, M. \"{O}nnheim, \emph{Propagation of Chaos for a Class of First Order Models with Singular Mean Field Interactions,} SIAM J. Math. Anal. 51(2019), 159-196.


\bibitem{CN} P.-E. Chaudry De Raynal, N. Frikha, \emph{Well-posedness for some non-linear SDEs and related PDE on the Wasserstein space}, arXiv:1811.06904, to appear in J. Math. Pures Appl..

\bibitem{DGW} H. Djellout, A. Guillin, L. Wu, \emph{Transportation cost-information inequalities and applications to random dynamical systems and diffusions,} Ann. Probab. 32(2004), 2702-2732.



\bibitem{FF} E. Fedrizzi, F. Flandoli, E. Priola, J. Vovelle, \emph{Regularity of stochastic kinetic equations,} Electron. J. Probab. 22(2017), 1-48.

\bibitem{FG} N. Fournier, A. Guillin, \emph{On the rate of convergence in Wasserstein distance of the empirical measure,} Probab. Theory Related Fields 162(2015), 707-738.

\bibitem{G} J. G\"{a}rtner, \emph{On the McKean-Vlasov limit for interacting diffusions,} Mathematische Nachrichten 137(1988), 197-248.

\bibitem{GKMPPT} C. Graham, T. Kurtz, S. M\'{e}l\'{e}ard, P. Protter, M. Pulvirenti, D. Talay, \emph{Probabilistic Models for Non-linear
Partial Differential Equations, Lecture Notes in Mathematics 1627,} Springer-Verlag (1996).

\bibitem{GW} A. Guillin, F.-Y. Wang, \emph{Degenerate Fokker-Planck equations: Bismut formula, gradient estimate and Harnack inequality,} J. Differential Equations 253(2012), 20-40.


\bibitem{HX} X. Huang,  \emph{Path-distribution dependent SDEs with singular coefficients,}  Electron. J. Probab. 26(2021), 1-21.


\bibitem{HLF}   X. Huang, W. Lv, \emph{Stochastic functional Hamiltonian system with singular coefficients,} Commun. Pur. Appl. Anal. 19(2020), 1257-1273.

\bibitem{HRW} X. Huang, M. R\"{o}ckner, F.-Y. Wang, \emph{Non-linear Fokker--Planck equations for probability measures on path space and path-distribution dependent SDEs,}  Discrete Contin. Dyn. Syst. 39(2019), 3017-3035.

\bibitem{HS} X. Huang, Y. Song, \emph{Well-posedness and regularity for distribution dependent SPDEs with singular drifts,} Nonlinear Anal. 203(2021), 112167.





\bibitem{JR} B. Jourdain, J. Reygner, \emph{Propagation of chaos for rank-based interacting diffusions and long time behaviour of a scalar quasilinear parabolic equation,}  Stoch. PDE: Anal. Comp. 1(2013), 455-506.


\bibitem{L} D. Lacker, \emph{On a strong form of propagation of chaos for McKean-Vlasov equations,} Electron. Commun. Probab. 23(2018), 1-11.

\bibitem{M} H. P. McKean, \emph{A class of Markov processes associated with nonlinear parabolic equations,} Proc Natl Acad Sci U S A, 56(1966), 1907-1911.

\bibitem{MD} L. Miclo, P. Del Moral,
\emph{Genealogies and Increasing Propagation of Chaos For Feynman-Kac and Genetic Models,}
Ann. Appl. Probab. 11(2001), 1166-1198.

\bibitem{MV} Yu. S. Mishura, A. Yu. Veretennikov, \emph{Existence and uniqueness theorems for solutions of McKean-Vlasov stochastic equations,} Theor. Probability and Math. Statist. 103(2020), 59-101.

\bibitem{NT} M. Nagasawat, H. Tanaka, \emph{Diffusion with Interactions and Collisions Between
Coloured Particles and the Propagation of Chaos,} Probab. Theory Related Fields 74(1987), 161-198.

\bibitem{Pin} M. S. Pinsker, \emph{Information and Information Stability of Random Variables and Processes,} Holden-Day, San Francisco, 1964.

\bibitem{RW}  P. Ren, F.-Y. Wang, \emph{Exponential convergence in entropy and Wasserstein for McKean-Vlasov SDEs,}  Nonlinear Anal. 206(2021), 112259.
\bibitem{RZ} M. R\"{o}ckner, X. Zhang, \emph{Well-posedness of distribution dependent SDEs with singular drifts,} Bernoulli 27(2021), 1131-1158.



\bibitem{RKF} S. T. Rachev, L. B. Klebanov, S. V. Stoyanov, F. J. Fabozzi, \emph{Glivenko-Cantelli Theorem and Bernstein-Kantorovich Invariance Principle,} The Methods of Distances in the Theory of Probability and Statistics. Springer, New York, 2013.

\bibitem{SZ} A.-S. Sznitman,   \emph{Topics in propagation of chaos,} In $``$\'Ecole d'\'Et\'e de Probabilit\'es de Sain-Flour XIX-1989", Lecture Notes in Mathematics  1464, p. 165-251, Springer, Berlin, 1991.

\bibitem{V} C. Villani,  \emph{Hypocoercivity,} Mem. Amer. Math. Soc. 202(2009).

\bibitem{Wbook} F.-Y. Wang, \emph{Harnack Inequality for Stochastic Partial Differential Equations,} Springer, New York, 2013.

\bibitem{FYW1} F.-Y. Wang, \emph{Distribution-dependent SDEs for Landau type equations,} Stochastic Process Appl. 128(2018), 595-621.

\bibitem{W2} F.-Y. Wang, \emph{Hypercontractivity and Applications for Stochastic Hamiltonian Systems,} J. Funct. Anal. 272(2017), 5360-5383.

\bibitem{WZ1} F.-Y. Wang, X. Zhang, \emph{Derivative formula and applications for degenerate diffusion semigroups,} J. Math. Pures Appl. 99(2013),726-740.

\bibitem{WZ} F.-Y. Wang, X. Zhang,  \emph{Degenerate SDE with H\"{o}lder-Dini Drift and Non-Lipschitz Noise Coefficient,} SIAM J. Math. Anal. 48(2016), 2189-2226.




\bibitem{Z1} X. Zhang, \emph{Stochastic flows and Bismut formulas for stochastic Hamiltonian systems,} Stochastic Process Appl. 120(2010), 1929-1949.

\bibitem{Z2} X. Zhang, \emph{Stochastic hamiltonian flows with singular coefficients,} Sci. China Math. 61(2018),1353-1384.


\bibitem{Z5} X. Zhang, \emph{Second order McKean-Vlasov SDEs and kinetic Fokker-Planck-Kolmogorov equations,} arXiv:2109.01273.



\end{thebibliography}

\end{document}